\documentclass[oneside,english]{amsart}

\usepackage{amssymb}
\usepackage{amscd}

\numberwithin{equation}{section} %% Comment out for sequentially-numbered
\numberwithin{figure}{section} %% Comment out for sequentially-numbered
\theoremstyle{plain}
\newtheorem*{thm*}{Theorem}
\theoremstyle{plain}
\newtheorem{thm}{Theorem}[section]
\theoremstyle{definition}
\newtheorem{defn}[thm]{Definition}
\theoremstyle{plain}

\theoremstyle{plain}
\newtheorem{prop}[thm]{Proposition}
\theoremstyle{plain}
\newtheorem{cor}[thm]{Corollary}
\theoremstyle{remark}

\theoremstyle{remark}
\newtheorem*{acknowledgement*}{Acknowledgement}

\begin{document}

\title[Generalized Helmholtz conditions]{Generalized Helmholtz
  conditions for non-conservative Lagrangian systems}

\author[Bucataru]{Ioan Bucataru} 
\address{Ioan Bucataru, Faculty of  Mathematics, Alexandru Ioan Cuza
  University \\ Ia\c si,  Romania} 
\urladdr{http://www.math.uaic.ro/\textasciitilde{}bucataru/}

\author[Constantinescu]{Oana Constantinescu} 
\address{Oana Constantinescu, Faculty of  Mathematics, Alexandru Ioan Cuza
  University \\ Ia\c si,  Romania} 
\urladdr{http://www.math.uaic.ro/\textasciitilde{}oanacon/}

\date{\today}

\begin{abstract}
In this paper we provide generalized Helmholtz conditions, in terms of a semi-basic
$1$-form, which characterize when a given system of second order ordinary
differential equations is equivalent to the Lagrange equations, for
some given arbitrary non-conservative forces. For the particular cases of dissipative or gyroscopic forces, these conditions, when expressed in terms of a multiplier matrix, reduce to
those obtained in \cite{MSC11}. When the
involved geometric structures are homogeneous with respect to the
fibre coordinates, we show how one can
further simplify the generalized Helmholtz conditions. We provide
examples where the proposed generalized Helmholtz conditions,
expressed in terms of a semi-basic $1$-form, can be integrated
and the corresponding Lagrangian and Lagrange equations can be found.
\end{abstract}



\maketitle

\section{Introduction}

The classic inverse problem of Lagrangian mechanics requires to find
the necessary and sufficient conditions, which are called
\emph{Helmholtz conditions}, such that a given system of second order
ordinary differential equations (SODE) is equivalent to the
Euler-Lagrange equations of some regular Lagrangian function. The
problem has a long history and the literature about the subject is vast.
There are various approaches to this problem, using different techniques and
mathematical tools, \cite{AT92, BD09,  CM91, Crampin81, CSMBP94, Krupkova97, MFVMR90, Sarlet82}.

In this work we discuss the inverse problem of Lagrangian systems with
non-conservative forces. Locally, the problem can be formulated as
follows. We consider a SODE in normal form
\begin{eqnarray}
\frac{d^2x^i}{dt^2} + 2G^i\left(x, \dot{x}\right)=0 \label{sode} \end{eqnarray}   
and an arbitrary covariant force field $\sigma_i(x, \dot{x})dx^i$. We 
provide necessary and sufficient conditions, which we call
\emph{generalized Helmholtz conditions}, for the existence of a Lagrangian $L$, such that the system \eqref{sode} is
equivalent to the Lagrange equations
\begin{eqnarray}
\frac{d}{dt}\left(\frac{\partial L}{\partial \dot{x}^i}\right) -
\frac{\partial L}{\partial x^i} = \sigma_i(x,\dot{x}). \label{lagrange_eq} 
\end{eqnarray} 
When the covariant forces are of dissipative or gyroscopic type, the
problem has been studied recently in \cite{CMS10, MSC11}. In these two
papers the authors provide generalized Helmholtz conditions, in terms of a multiplier matrix, for a SODE
\eqref{sode} to represent Lagrange equations with
non-conservative forces of dissipative or gyroscopic type.

The structure of the paper is as follows. In Section \ref{section_gf}
we use the Fr\"olicher-Nijenhuis formalism \cite{FN56, GM00}
to provide a geometric framework associated to a given system
\eqref{sode}. This framework includes: a nonlinear connection,
dynamical covariant derivative and curvature type tensors. In Section \ref{section_ghc} we use this
geometric setting to reformulate the inverse problem of Lagrangian
systems with non-conservative forces. The main contribution of this paper is to provide, in
Theorem \ref{thm2}, generalized Helmholtz conditions in terms of
semi-basic $1$-forms for the most general case of the problem. Such semi-basic $1$-form will represent the
Poincar\'e-Cartan $1$-form of the sought after Lagrangian. In the particular case when the covariant force field $\sigma$ is zero, the
generalized Helmholtz conditions $(GH_1) - (GH_3)$ of Theorem \ref{thm2}
reduce to, and simplify, the Helmholtz obtained in \cite[Theorem 4.1]{BD09}. 

In the next two sections we show that the proposed generalized Helmholtz
conditions $(GH_1) - (GH_3)$, of Theorem \ref{thm2}, reduce to
those obtained in \cite{CMS10, MSC11},
which were expressed in terms of a multiplier matrix, for the
particular case of dissipative or gyroscopic forces. Theorem
\ref{thm3} provides three equivalent sets of conditions, in terms of
semi-basic $1$-forms, for a SODE to be of dissipative type. One 
advantage of formulating the generalized Helmholtz conditions in terms
of forms is discussed in Proposition \ref{cor2}, where we study the formal
integrability of such conditions. An important consequence of
Proposition \ref{cor2} is that any SODE on a $2$-dimensional manifold is
of dissipative type. Theorem \ref{thm4} provides two equivalent sets of
generalized Helmholtz conditions, in terms of semi-basic $1$-forms,
which characterize Lagrangian systems of gyroscopic type. 

In section \ref{section_hghc} we discuss the inverse problem of
Lagrangian systems with non-conservative forces, when all the involved geometric objects
are homogeneous with respect to the velocity coordinates. Within this context, in Theorem \ref{thm5}, we
prove that one generalized Helmholtz condition is a consequence
of the other two. When the covariant force field is zero, depending on the degree of
homogeneity, the problem reduces to the Finsler metrizability problem
or the projective metrizability problem. 

In the last section we show how the techniques developed throughout
the paper can be used to discuss various examples. For these
examples, the generalized Helmholtz conditions, expressed in terms of
a semi-basic $1$-form, can be
integrated and therefore we can find the corresponding Lagrangian and
Lagrange equations. The examples we analyse consist of
non-variational projectively metrizable sprays that are of dissipative
type and a class of gyroscopic semisprays. For each of the two
examples, the techniques used in the proof of Theorems \ref{thm3} and
\ref{thm4} are very useful for integrating the corresponding
generalized Helmholtz conditions.

\section{The geometric framework} \label{section_gf}

\subsection{A geometric setting for semisprays}

In this section, we use the Fr\"olicher-Nijenhuis formalism
\cite{FN56} to associate a geometric setting to a given system of second order ordinary differential equations, \cite{BCD11, BD09, Grifone72, GM00}.  

For an $n$-dimensional smooth manifold $M$, denote by $TM$ its tangent
bundle. Local coordinates $(x^i)$ on $M$ induce local coordinates $(x^i,
y^i)$ on $TM$. The set smooth functions on $M$ will be denoted by
$C^{\infty}(M)$, while the set of smooth vector fields on $M$
will be denoted by $\mathfrak{X}(M)$. 

Consider $\mathbb{C}\in \mathfrak{X}(TM)$ the Liouville (dilation) vector field and
$J$ the tangent structure (vertical endomorphism). Throughout this
paper we use the summation convention over covariant and
contravariant repeated indices. With this convention, the Liouville
vector field and the tangent structure are locally given by:
\begin{eqnarray*}
\mathbb{C}=y^i\frac{\partial}{\partial y^i}, \quad
J=\frac{\partial}{\partial y^i}\otimes dx^i. \end{eqnarray*}
We consider the regular $n$-dimensional vertical distribution, $VTM: u\in TM \to
V_uTM=\operatorname{Ker} d_u\pi\subset T_uTM$. The forms dual to the
vertical vector fields will play an important role in this work. These
are semi-basic (vector valued) forms on $TM$, with respect to the canonical
projection $\pi$. 

In order to develop a geometric setting, we will make use of the Fr\"olicher-Nijenhuis
formalism. Within this formalism one can identify derivations to
vector valued forms on $TM$, \cite{FN56, GM00}. 

For a vector valued $p$-form $P$ on
$TM$, we denote by $i_P:\Lambda^k(TM)\to \Lambda^{k+p-1}(TM)$ the 
derivation of degree $(p-1)$, given by 
\begin{eqnarray*}
i_P\alpha(X_1,..., X_{k+p-1})= \frac{1}{p!(k-1)!}\sum_{\sigma\in
  S_{k+p-1}} \operatorname{sign}(\sigma)\alpha\left(P(X_{\sigma(1)}, ...,
  X_{\sigma(p)}), X_{\sigma(p+1)}, ..., X_{\sigma(k+p-1)}\right), \end{eqnarray*}
where $S_{k+p-1}$ is the permutation group of $\{1,.., k+p-1\}$. We
denote by $d_P:\Lambda^k(TM)\to \Lambda^{k+p}(TM)$ the 
derivation of degree $p$, given by 
\begin{eqnarray*}
  d_P=[i_P, d]=i_P\circ d - (-1)^{p-1} d\circ i_P. \end{eqnarray*}
For two vector valued forms $K$ and $P$ on $TM$, of degrees $k$ and
$p$, we consider the Fr\"olicher-Nijenhuis bracket $[K, P]$, which is
the vector valued $(k+p)$-form, uniquely determined by
\begin{eqnarray}
d_{[K, P]}=[d_K, d_P]=d_k\circ d_P - (-1)^{kp}d_P \circ
d_K. \label{dkl}\end{eqnarray}
For various commutation formulae, within the Fr\"olicher-Nijenhuis
formalism, we will use the Appendix A of the book \cite{GM00}.
We will use some vector valued forms on $TM$ to associate a
differential calculus on $TM$. Directly from the definition of the tangent
structure $J$ it follows that $[J,J]=0$ and according to formula
\eqref{dkl} it follow that $d_J^2=0$. Therefore, any $d_J$-exact form
is $d_J$-closed and according to a Poincar\'e-type Lemma
\cite{Vaisman73}, any $d_J$-closed form is locally $d_J$-exact. The
derivation $d_J$ corresponds to the operator $\dot{d}$ used in \cite{Klein62}. 

A \emph{semispray}, or a second order vector field, is a globally
defined vector field on $TM$, $S\in \mathfrak{X}(TM)$, that satisfies
$JS=\mathbb{C}$. Locally, it can be expressed as:
\begin{eqnarray}
S=y^i\frac{\partial}{\partial x^i} - 2G^i(x,y)\frac{\partial}{\partial
  y^i}. \label{semispray} \end{eqnarray} 
A curve $c: t\in I \subset \mathbb{R} \to c(t)=(x^i(t))\in M$ is
a \emph{geodesic} of the semispray $S$ if $c': t\in I \subset \mathbb{R} \to c'(t)=(x^i(t), dx^i/dt)\in
TM$, is an integral curve of $S$, which means that it satisfies \eqref{sode}.

A semispray $S$ induces a horizontal and a vertical projector, $h$ and $v$
that are given by, \cite{Grifone72}, 
\begin{eqnarray*}
h=\frac{1}{2}\left(\operatorname{Id} - [S,J]\right), \quad
v=\frac{1}{2}\left(\operatorname{Id} + [S,J]\right). \end{eqnarray*}
Locally, the two projectors $h$ and $v$ can be expressed as
\begin{eqnarray*}
h=\frac{\delta}{\delta x^i}\otimes dx^i, \quad
v=\frac{\partial}{\partial y^i}\otimes \delta y^i, \quad
\frac{\delta}{\delta x^i}=\frac{\partial}{\partial x^i} - N^j_i
\frac{\partial}{\partial y^j}, \quad \delta y^i=dy^i+N^i_j dx^j, \quad
N^i_j=\frac{\partial G^i}{\partial y^j}.
\end{eqnarray*}

The Fr\"olicher-Nijenhuis bracket $[S,h]$ induces
two geometric structures. One is the \emph{almost complex structure},
${\mathbb F}$, and the other one is the \emph{Jacobi endomorphism}, $\Phi$,
\begin{eqnarray}
{\mathbb F}=h\circ[S,h]-J, \quad \Phi=v\circ
[S,h]. \label{fphi} \end{eqnarray}
Locally, the almost complex structure can be expressed as follows
\begin{eqnarray*}
{\mathbb F}=\frac{\delta}{\delta x^i} \otimes \delta y^i -
\frac{\partial}{\partial y^i}\otimes dx^i. \end{eqnarray*}
The Jacobi endomorphism has the following
local expression
\begin{eqnarray}
\Phi=R^i_j\frac{\partial}{\partial y^i} \otimes dx^j, \quad
R^i_j=2\frac{\partial G^i}{\partial x^j} - S(N^i_j)- N^i_r
N^r_j. \label{rij} \end{eqnarray}
The horizontal distribution induced by a semispray is, in general,
non-integrable. The obstruction to its integrability is given by the
curvature tensor
\begin{eqnarray}
R=\frac{1}{2}[h,h]=\frac{1}{2} R^i_{jk}\frac{\partial}{\partial y^i}\otimes
dx^j\wedge dx^k, \quad R^i_{jk}= \frac{\delta
    N^i_j}{\delta x^k} - \frac{\delta N^i_k}{\delta
    x^j}. \label{curvature} \end{eqnarray}
The Jacobi endomorphism $\Phi$ and the curvature tensor $R$ are
closely related by
\begin{eqnarray}
3R=[J,\Phi], \quad 3R^i_{jk} = \frac{\partial R^i_j}{\partial y^k}-
\frac{\partial R^i_k}{\partial y^j}. \label{rphi} \end{eqnarray}
From the above first formula and \eqref{dkl} it follows the
commutation formula
\begin{eqnarray}
[d_J, d_{\Phi}]=3d_R. \label{djphi} \end{eqnarray}

We introduce now the \emph{dynamical covariant derivative}, $\nabla$, associated
to a semispray, following the approach from \cite{BCD11, BD09}. For $f\in C^{\infty}(TM)$ and $X\in {\mathfrak
  X}(TM)$, we define 
\begin{eqnarray}
\nabla f= Sf, \quad \nabla X=h[S,hX] +
v[S,vX]. \label{nabla1} \end{eqnarray}
Using the formulae \eqref{fphi}, we can write the action of $\nabla$ on vector fields as
\begin{eqnarray}
\nabla = h\circ {\mathcal L}_S \circ h + v\circ {\mathcal L}_S \circ v
= {\mathcal L}_S + {\mathbb F}+J - \Phi. \label{nabla2} \end{eqnarray}
Therefore, the action of $\nabla$ on the exterior algebra of $TM$ is
given by 
\begin{eqnarray}
\nabla = {\mathcal L}_S -i_{{\mathbb F}+J -
  \Phi}. \label{nabla3} \end{eqnarray} 
The following commutation formula
can be shown using items iii) and iv) of \cite[Theorem 3.5]{BD09}
\begin{eqnarray}
[d_J, \nabla] = d_h + 2i_R. \label{djn}
\end{eqnarray}
For more properties of the dynamical covariant derivative and
some commutation formulae with other geometric structures, we
refer to \cite[Section 3.2]{BD09}. 
 
\subsection{Lagrange systems and non-conservative covariant forces.}

Consider $L: TM \to\mathbb{R}$ a Lagrangian, which is a smooth function
on $TM$ whose Hessian with respect to the fibre coordinates 
\begin{eqnarray} 
g_{ij}=\frac{\partial^2 L}{\partial y^i\partial y^j} \label{gij} \end{eqnarray}
is nontrivial. We say that $L$ is a \emph{regular Lagrangian} if the
Poincar\'e-Cartan $2$-form $dd_JL$ is a symplectic form on $TM$. Locally,
the regularity condition of a Lagrangian $L$ is equivalent to the fact
that the Hessian \eqref{gij}  of $L$ has maximal rank $n$ on $TM$. 

For an arbitrary semispray $S$ and a Lagrangian $L$, the following $1$-form (called the
\emph{Euler-Lagrange} $1$-form, or \emph{the Lagrange differential} in \cite{Tulczyjew76}) is a semi-basic $1$-form:
\begin{eqnarray}
\label{elform} \delta_SL & =& {\mathcal L_S}d_JL - dL= d_J  {\mathcal
  L_S} L - 2d_hL= \nabla d_JL - d_hL \\ & = &
\left\{S\left(\frac{\partial L}{\partial y^i}\right) - \frac{\partial
    L}{\partial x^i}\right\}dx^i = \left\{\frac{\partial
    S(L)}{\partial y^i} - 2\frac{\delta L}{\delta x^i}\right\}dx^i  = 
\left\{\nabla \left(\frac{\partial L}{\partial
      y^i}\right) - \frac{\delta L}{\delta
    x^i}\right\}dx^i. \nonumber \end{eqnarray}
The inverse problem of Lagrangian mechanics requires, for a given
semispray $S$, to decide wether or not there exists a Lagrangian $L$
with vanishing Lagrange differential, which means
$\delta_SL=0$. In this case we will say that the semispray $S$ is
\emph{Lagrangian}. Locally it means that the solutions of the system \eqref{sode}
are among the solutions of the Euler-Lagrange equations of some Lagrangian
$L$. Necessary and sufficient conditions for the 
existence of such Lagrangian are called Helmholtz conditions and were
expressed in terms of a multiplier matrix \cite{CMS13, CSMBP94, Krupkova97, Sarlet82}, a semi-basic
$1$-form \cite{BD09, CMS13}, or a $2$-form \cite{AT92, Crampin81,
  Klein62, MFVMR90}. 

In this work we study the more general problem, when for a given
semispray $S$ and a semi-basic $1$-form $\sigma$, we ask for the
existence of a Lagrangian $L$, whose Lagrange differential is
$\sigma$. 

\begin{defn} \label{def_lt}
Consider $S$ a semispray and $\sigma\in \Lambda^1(TM)$ a semi-basic
$1$ form. We say that $S$ is of \emph{Lagrangian type with covariant force
  field } $\sigma$ if there exists a (locally defined) Lagrangian $L$
such that $\delta_SL=\sigma$. \end{defn}
The above definition expresses the fact that the solutions of the system \eqref{sode}
are among the solutions of the Lagrange equations \eqref{lagrange_eq}.
If the Lagrangian $L$, which we search for, is regular, then the two systems
\eqref{sode} and \eqref{lagrange_eq} are equivalent. 

There is an important aspect of Definition \ref{def_lt} that we want to
emphasize, if we do not make any requirement about the covariant force
field $\sigma$. For an arbitrary semispray $S$ there is always a
Lagrangian $L$ and a semi-basic $1$-form $\sigma$ such that
$\delta_SL=\sigma$.  This case corresponds to the \emph{semi-variational
  equations} studied in \cite[Section 2]{Rossi14}. In our analysis, we start with a given
semispray $S$  and a given semi-basic $1$-form $\sigma$ on $TM$ and search
for a Lagrangian $L$ such that $\delta_SL=\sigma$. More exactly, for a given semispray $S$
and a semi-basic $1$-form $\sigma$, we provide necessary and
sufficient conditions, which we will call \emph{generalized Helmholtz
  conditions}, for the existence of a semi-basic $1$-form $\theta$
that represents the Poincar\'e-Cartan $1$-form of a Lagrangian $L$
such that $\delta_SL=\sigma$.

\section{Generalized Helmholtz conditions} \label{section_ghc}

In this section we provide necessary and sufficient conditions
for a given semispray $S$ to be of Lagrangian type with a given
covariant force field $\sigma$. These
conditions, which we will refer to as \emph{generalized Helmholtz
  conditions}, will be expressed in terms of a semi-basic $1$-form. We will prove that for some particular cases of the
covariant force field (dissipative and gyroscopic) the generalized
Helmholtz conditions reduce to those obtained in \cite{MSC11} in terms
of a multiplier matrix.  

Throughout this work, we make the following assumption about the
semi-basic $1$-form $\theta$ that will be involved in expressing the
generalized Helmholtz conditions. We say that a semi-basic $1$-form
$\theta=\theta_i(x,y)dx^i$ is \emph{non-trivial} if the matrix
$g_{ij}={\partial \theta_i}/{\partial y^j}$ is non-trivial. If
$\theta=d_JL$ is the Poincar\'e-Cartan $1$-form of some function $L$,
then $\theta$ is non-trivial if and only if $L$ is a Lagrangian.    

\begin{thm} \label{thm1}
Consider $S$ a semispray and $\sigma\in \Lambda^1(TM)$ a semi-basic
$1$-form. The semispray $S$ is of \emph{Lagrangian type with
covariant force field} $\sigma$ if
and only if there exists a non-trivial, semi-basic $1$-form $\theta \in
\Lambda^1(TM)$ such that ${\mathcal L}_S\theta - \sigma$ is a closed
$1$-form on $TM$.
\end{thm}
\begin{proof}
For the direct implication, from Definition \ref{def_lt}, it follows
that there exists a Lagrangian $L$ such that ${\mathcal L}_Sd_JL-dL=\sigma$. We
take $\theta=d_JL$, the Poincar\'e-Cartan $1$-form of $L$. It follows
that  ${\mathcal L}_S\theta - \sigma=dL$, which is exact and hence
it is a closed $1$-form. Since $L$ is a Lagrangian, we have that
the semi-basic $1$-form $\theta$ is non-trivial.

For the converse implication, we assume that there exists $\theta \in
\Lambda^1(TM)$ a non-trivial, semi-basic $1$-form, such that ${\mathcal L}_S\theta
- \sigma$ is a closed $1$-form on $TM$. Therefore, there exists a
(locally defined) function $L$ on $TM$ such that 
\begin{eqnarray}
{\mathcal L}_S\theta - \sigma=dL. \label{lsts} \end{eqnarray} We apply
$i_J$ to both sides of this
formula. In the right hand side we have $i_JdL=d_JL$. We evaluate now
the left hand side. Since
$\theta$ and $\sigma$ are semi-basic $1$-forms, it follows that
$i_J\theta = i_J\sigma=0$. For a vector valued $1$-form $K$ on $TM$,
we use the commutation formula, see
\cite[A.1, page 205]{GM00}, 
\begin{eqnarray}
i_K{\mathcal L}_S={\mathcal L}_Si_K + i_{[K, S]}. \label{ills} 
\end{eqnarray}
If $K=J$, the tangent structure, we use above formula and $[J,S]=h-v$. It follows that
$i_J{\mathcal L}_S\theta = i_{[J, S]}\theta = i_h\theta=\theta$ and
hence $\theta=d_JL$. Therefore, the non-trivial, semi-basic $1$ form $\theta$ is the Poincar\'e-Cartan $1$-form
of $L$, and hence $L$ is a Lagrangian function. We replace this in formula \eqref{lsts} and
obtain that the semispray $S$ is of Lagrangian type with the 
covariant force field $\sigma$.   
\end{proof}
In Theorem \ref{thm1}, if we search for a regular Lagrangian $L$, then the
corresponding semi-basic $1$-form $\theta$ has to be regular as well, in the
following sense. A semi-basic $1$-form $\theta\in \Lambda^1(TM)$ is said to be \emph{regular}
if $d\theta$ is a symplectic $2$-form.  If
$\theta=d_JL$ is the Poincar\'e-Cartan $1$-form of some function $L$,
then the regularity condition of $\theta$ is equivalent to the regularity of the
Lagrangian $L$.    

Next theorem provides necessary and sufficient conditions for the
existence of the semi-basic $1$-form, which was discussed in
Theorem \ref{thm1}, using the differential operators associated to a
given semispray. 

\begin{thm} \label{thm2}
A semispray $S$ is of \emph{Lagrangian type with covariant force field} $\sigma$ if
and only if there exists a non-trivial, semi-basic $1$-form $\theta \in
\Lambda^1(TM)$ such that the following
\emph{generalized Helmholtz conditions} are satisfied:
\begin{itemize}
\item[$(GH_1)$] $d_J\theta=0$;
\item[$(GH_2)$] $d_{\Phi}\theta=\frac{1}{2}\nabla d_J\sigma - d_h\sigma$; 
\item[$(GH_3)$] $\nabla d_v\theta = d_v\sigma - \frac{1}{2}i_{{\mathbb F}+J}d_J\sigma$. 
\end{itemize} 
 \end{thm}
\begin{proof}
We fix a semispray $S$ and  a semi-basic $1$-form $\sigma \in
\Lambda^1(TM)$. According to Theorem \ref{thm1} we have that
$S$ is of Lagrangian type with covariant force field $\sigma$ if
and only if there exists a non-trivial, semi-basic $1$-form $\theta \in
\Lambda^1(TM)$ such that 
\begin{eqnarray}
{\mathcal L}_Sd\theta=d\sigma. \label{lsdts}
\end{eqnarray}
We will prove now that formula \eqref{lsdts} is equivalent to the three
generalized Helmholtz conditions $(GH_1)-(GH_3)$.

For the direct implication, we consider $\theta$ a
non-trivial, semi-basic $1$-form on $TM$ that satisfies formula \eqref{lsdts}. We
apply to both sides of this formula the derivation $i_J$. Using
commutation formula \eqref{ills} for $K=J$ and the fact that
$i_{[J,S]}d\theta = i_{h-v}d\theta = i_{2h-\operatorname{Id}}d\theta =
2i_hd\theta - 2d\theta = 2d_h\theta$, we obtain 
\begin{eqnarray}
{\mathcal L}_Sd_J\theta +2d_h\theta =
d_J\sigma. \label{lsdjt}\end{eqnarray}
We apply again the derivation $i_J$ to both sides of the above
formula and we use that $d_h\theta$ and $d_J\sigma$ are semi-basic
$2$-forms, which implies that $i_Jd_h\theta=0$ and
$i_Jd_J\sigma=0$. Using again the commutation formula for $i_J$ and
${\mathcal L}_S$, it follows that $i_{[J,S]}d_J\theta=0$. Since
  $i_{[J,S]}d_J\theta=i_{h-v}d_J\theta=2d_J\theta$ we obtain that
    $d_J\theta=0$, which is the first generalized Helmholtz condition
    $(GH_1)$. We replace this in formula \eqref{lsdjt} and obtain the formula
\begin{eqnarray} d_h\theta = \frac{1}{2}d_J\sigma. \label{dhtdjs} \end{eqnarray} 
Since $\theta$ is a semi-basic $1$-form, it follows that
$i_h\theta=\theta$ and $i_v\theta=0$. Therefore we have 
\begin{eqnarray*}
d\theta=2d\theta - d\theta = i_{\operatorname{Id}}d\theta - d\theta =
i_hd\theta +i_vd\theta - d\theta = d_h\theta +
d_v\theta. \end{eqnarray*}
The condition $d_J\theta=0$ reads $d_J\theta(X, Y)=d\theta(JX, Y)
+d\theta(X, JY) =0$, for all $X, Y\in \mathfrak{X}(TM)$. Hence, for any two vertical vector fields $V, W$ on $TM$, we have
$d\theta(V, W)=0$.

In order to show that the next two generalized Helmholtz conditions are
satisfied, we will prove first that 
\begin{eqnarray}
i_{{\mathbb F}+J}d_v\theta=0. \label{ifdvt} \end{eqnarray}
Consider $X_1, Y_1 \in {\mathfrak
  X}(TM)$. There exist $X_2, Y_2 \in {\mathfrak
  X}(TM)$ such that $({\mathbb F}+J)(X_1)=hX_2$ and $({\mathbb
  F}+J)(Y_1)=hY_2$. Moreover, if we compose to the left these two
equalities with the tangent structure $J$ and use the fact that
$J\circ {\mathbb F}= v$ and $J\circ h = J$, we obtain $vX_1=JX_2$ and $vY_1=JY_2$. Using
these equalities we have
\begin{eqnarray*}
\left(i_{{\mathbb F}+J}d_v\theta\right)(X_1, Y_1) & = & 
d\theta\left(({\mathbb F}+J)(X_1), vY_1\right) + d\theta\left(vX_1,
  ({\mathbb F}+J)(Y_1)\right) \\ & = &  d\theta\left(hX_2, JY_2\right) + d\theta\left(JX_2,
  hY_2\right) = d_J\theta(X_2, Y_2)=0. \end{eqnarray*}
We apply the derivation $i_h$ to both sides of formula
\eqref{lsdts} and use the commutation rule \eqref{ills} for $K=h$,
which gives
\begin{eqnarray*}
{\mathcal L}_Si_hd\theta + i_{[h,S]}d\theta =
i_hd\sigma. \end{eqnarray*}
We use the fact that $\theta$ and $\sigma$ are semi-basic forms, which
implies that $i_hd\theta=d_h\theta - d\theta$ and
$i_hd\sigma=d_h\sigma - d\sigma$, and formula \eqref{lsdts} again to
obtain
\begin{eqnarray}
{\mathcal L}_Sd_h\theta + i_{[h,S]}d\theta =
d_h\sigma. \label{lsdht}\end{eqnarray}
From the two formulae \eqref{fphi} we have $[h,S]= -{\mathbb F}- J
-\Phi$. Now using formula \eqref{ifdvt} we obtain $i_{[h,S]}d\theta = -i_{{\mathbb F}+J}d\theta
-i_{\Phi}d\theta =  -i_{{\mathbb F}+J}d_h\theta
-d_{\Phi}\theta$. With this formula we go back to \eqref{lsdht},
where we use the fact that ${\mathcal
  L}_Sd_h\theta - i_{{\mathbb F}+J}d_h\theta =\nabla d_h\theta$ and hence
\begin{eqnarray}
\nabla d_h\theta - d_{\Phi}\theta =
d_h\sigma. \label{ndht} \end{eqnarray}
If we make use of formula \eqref{dhtdjs} to
substitute $d_h\theta$, we obtain that the second generalized Helmholtz
condition $(GH_2)$ is true as well. Using formula \eqref{nabla3} we
obtain 
\begin{eqnarray*}
{\mathcal L}_Sd\theta = \nabla d\theta +i_{{\mathbb F}+J}d\theta -
i_{\Phi}d\theta. \label{lsdnd} \end{eqnarray*}
If we replace this in \eqref{lsdts}, we obtain 
\begin{eqnarray*}
\nabla d_h\theta + \nabla d_v\theta + i_{{\mathbb F}+J}d_h\theta -
d_{\Phi}\theta = d_h\sigma + d_v\sigma. \end{eqnarray*}
In the above formula we use \eqref{ndht} and formula \eqref{dhtdjs} and
obtain that the last generalized Helmholtz condition $(GH_3)$ is true
as well.

We will prove now the converse, which means that the three generalized
Helmholtz conditions $(GH_1)-(GH_3)$ imply the condition
\eqref{lsdts}.  We will prove first that the existence of a non-trivial, semi-basic $1$-form
$\theta$ that satisfies the three conditions $(GH_1)-(GH_3)$ implies the
existence of a non-trivial, semi-basic $1$-form $\tilde{\theta}$ that satisfies
$(GH_1)-(GH_3)$ and \eqref{dhtdjs} as well. 

Consider $\theta$ a non-trivial, semi-basic $1$-form that satisfies the 
generalized Helmholtz conditions $(GH_1)-(GH_3)$. We apply the derivation $d_J$ to both sides of formula $(GH_2)$ to obtain
\begin{eqnarray}
d_Jd_{\Phi}\theta = \frac{1}{2}d_J\nabla d_J\sigma -
d_Jd_h\sigma. \label{djgh2} \end{eqnarray}
We evaluate first each of the two sides of the above formula. For the left hand side, using the
commutation formula \eqref{djphi}, as well as the fact that $d_J\theta=0$, we have 
\begin{eqnarray*}
 d_Jd_{\Phi}\theta=d_{[J,\Phi]}\theta = 3d_R\theta
 =3d_hd_h\theta. \end{eqnarray*}
Using the commutation formula \eqref{djn}, the fact that $d_J^2=0$ and the
fact that $i_Rd_J\sigma=0$, we can express the first term of
the right hand side of formula \eqref{djgh2} as
\begin{eqnarray*}
d_J\nabla d_J\sigma = d_hd_J\sigma. \end{eqnarray*}
Since $[J,h]=0$ it follows that $d_hd_J+d_Jd_h=0$. Now, if we replace
everything in both sides of formula \eqref{djgh2} we obtain
$ 3d_hd_h\theta = 3 d_hd_J\sigma/2$, 
which can be further written as 
\begin{eqnarray}
d_h\left(d_h\theta -
  \frac{1}{2}d_J\sigma\right)=0. \label{basic1} \end{eqnarray}
If we apply the derivation $d_J$ to both sides of formula $(GH_3)$ we obtain
\begin{eqnarray}
d_J\nabla d_v\theta = d_Jd_v\sigma - \frac{1}{2}d_Ji_{{\mathbb F}+J}
d_J\sigma. \label{djgh3} \end{eqnarray}
To evaluate the left hand side of formula \eqref{djgh3} we use the
commutation formula \eqref{djn}
\begin{eqnarray*}
d_J\nabla d_v\theta =\nabla d_Jd_v\theta + d_hd_v\theta +
2i_Rd_v\theta. \end{eqnarray*} 
We use the fact that $[J,v]=0$, which implies that
$d_Jd_v\theta+d_vd_J\theta=0$, to obtain that $d_Jd_v\theta=0$. Since
$2R=[h,h]=[h, \textrm{Id}-v]=-[h,v]$ it follows that
\begin{eqnarray*} d_hd_v\theta+d_vd_h\theta = d_{[h,v]}= -2d_R\theta =
  -2i_Rd_v\theta.
\end{eqnarray*}
We use all these calculations to write the left hand side of formula
\eqref{djgh3} as 
\begin{eqnarray*}
d_J\nabla d_v\theta = - d_vd_h\theta. \end{eqnarray*}
Finally, we have to evaluate the right hand side of formula
\eqref{djgh3}. For its second term, we will use the following
commutation formula, \cite[A.1, page 205]{GM00}, for two vector valued
$1$-forms $K$ and $P$
\begin{eqnarray}
i_Kd_P=d_Pi_K+d_{P\circ K}-i_{[K,P]}. \label{ikdl}
\end{eqnarray}
For $P=J$ and $K={\mathbb F}+J$ we have 
\begin{eqnarray*}
i_{{\mathbb F}+J}d_Jd_J\sigma = d_J i_{{\mathbb F}+J}d_J \sigma +
d_{J\circ ({\mathbb F}+J)} d_J\sigma - i_{[{\mathbb F}+J, J]}
d_J\sigma. \end{eqnarray*}
Since $J\circ ({\mathbb F}+J)=v$, $[{\mathbb F}+J, J] = [{\mathbb F},
J] = -R$ and $i_{[{\mathbb F}+J, J]} d_J\sigma = -i_Rd_J\sigma=0$ we obtain 
\begin{eqnarray*}
d_J i_{{\mathbb F}+J}d_J \sigma = -d_vd_J\sigma. \end{eqnarray*}
Using these calculations, we can write the right hand side of formula
\eqref{djgh3} as 
\begin{eqnarray*}
d_Jd_v\sigma - \frac{1}{2}d_Ji_{{\mathbb F}+J}
d_J\sigma = - d_vd_J\sigma + \frac{1}{2} d_vd_J\sigma = -\frac{1}{2}
d_vd_J\sigma  \end{eqnarray*} 
It follows that one can write formula \eqref{djgh3} as
\begin{eqnarray}
d_v\left(d_h\theta - \frac{1}{2}d_J\sigma
\right)=0. \label{basic2} \end{eqnarray}
From the two formulae \eqref{basic1} and \eqref{basic2}, we obtain
\begin{eqnarray} d\left(d_h\theta - \frac{1}{2}d_J\sigma
\right)=0. \label{basic3} \end{eqnarray}
Using the above formula, there exists a locally
defined basic $1$-form $\beta$ such that 
\begin{eqnarray} d_h\theta - \frac{1}{2}d_J\sigma= d\beta. \label{basic4} \end{eqnarray}
The semi-basic $1$-form $\tilde{\theta}=\theta-\beta$ satisfies all
three generalized Helmholtz condition $(GH_1)-(GH_3)$ and formula
\eqref{dhtdjs} as well. Since $\beta$ is a basic $1$-form, we have
that the semi-basic $1$-form $\tilde{\theta}$ is non-trivial if and
only if the semi-basic $1$-form $\theta$ is non-trivial.

For this non-trivial, semi-basic $1$-form $\tilde{\theta}$, we will prove that
formula \eqref{lsdts} is true. Formula \eqref{ifdvt}, which we proved in
the first part of our proof, is still true for $\tilde{\theta}$ since
for this we only need that $\tilde{\theta}$ is a semi-basic $1$-form that
satisfies $d_J\tilde{\theta}=0$. Using formula \eqref{nabla3} we
obtain 
\begin{eqnarray}
{\mathcal L}_Sd\tilde{\theta} = \nabla d_h\tilde{\theta}+ \nabla
d_v\tilde{\theta} + i_{{\mathbb F}+J}d_h\tilde{\theta} -
d_{\Phi}\tilde{\theta}. \label{lsdtext}\end{eqnarray} 
In the right hand side of formula \eqref{lsdtext} we replace $d_h\tilde{\theta}$,
$\nabla d_v\tilde{\theta}$ and $d_{\Phi}\tilde{\theta}$ in terms of
$\sigma$, from \eqref{dhtdjs} and the conditions $(GH_2)-(GH_3)$. It follows that formula \eqref{lsdts} is true. 
\end{proof}

In the absence of the exterior force, which means that $\sigma=0$, the
generalized Helmholtz conditions of Theorem \ref{thm2} reduce to
the Helmholtz conditions in \cite[Theorem 4.1]{BD09}. 
\begin{cor} \label{cor1}
A semispray $S$ is Lagrangian if and only if there exists a
non-trivial, semi-basic $1$-form $\theta$ that satisfies the following Helmholtz conditions
\begin{itemize}
\item[$(H_1)$] $d_J\theta=0$;
\item[$(H_2)$] $d_{\Phi}\theta=0$;
\item[$(H_3)$] $\nabla d_v\theta=0$.
\end{itemize}
\end{cor}
In \cite[Theorem 4.1]{BD09} there is an extra condition $d_h\theta=0$
that has been used. As we have seen in the proof of the second part of Theorem
\ref{thm2}, it can be shown that the three Helmholtz conditions $(H_1) - (H_3)$
imply that $d_h\theta=d\beta$, for a basic $1$-form
$\beta$. Therefore, the new semi-basic $1$-form $\tilde{\theta}=\theta-\beta$ satisfies the three Helmholtz
conditions $(H_1) - (H_3)$ as well as the fourth condition
$d_h\tilde{\theta}=0$, which was used in \cite[Theorem 4.1]{BD09}.    

We will provide now a local description of the three generalized
Helmholtz conditions $(GH_1) - (GH_3)$. Consider $S$ a semispray, locally given by formula
\eqref{semispray}, and let $\theta=\theta_idx^i$, $\sigma=\sigma_idx^i$
be two semi-basic $1$-forms on $TM$. We have  
\begin{eqnarray*}
d_v\theta & = &  g_{ij}\delta y^j \wedge dx^i, \quad
g_{ij}:=\frac{\partial \theta_i}{\partial y^j}. \\ d_J\theta & = & \frac{1}{2}\left(g_{ij}-g_{ji}\right) dx^j\wedge dx^i, \quad
d_{\Phi}\theta  = \frac{1}{2}\left(g_{ik}R^k_j-g_{jk}R^k_i \right)
dx^j\wedge dx^i. 
\end{eqnarray*}
The generalized Helmholtz conditions $(GH_1) - (GH_3)$ can be expressed locally as
follows
\begin{itemize}
\item[$(LGH_1)$] $g_{ij}=g_{ji}$. Due to the above definition for
  $g_{ij}$, we also  have $\displaystyle\frac{\partial g_{ij}}{\partial y^k} = \frac{\partial g_{ik}}{\partial y^j}$.
\item[$(LGH_2)$] $g_{ik}R^k_j-g_{jk}R^k_i=\displaystyle
  \frac{1}{2}\nabla \left( \frac{\partial
    \sigma_i}{\partial y^j} - \frac{\partial
    \sigma_j}{\partial y^i}\right) - \left( \frac{\delta
    \sigma_i}{\delta x^j} - \frac{\delta 
    \sigma_j}{\delta x^i}\right)$.
\item[$(LGH_3)$] $\nabla g_{ij}=\displaystyle\frac{1}{2}\left( \frac{\partial
    \sigma_i}{\partial y^j} + \frac{\partial
    \sigma_j}{\partial y^i}\right)$.
\end{itemize}
Condition $(LGH_3)$ and the local expression of $d_h\theta=0$ appear also in \cite[Theorem
3.1]{BM07}, as conditions that uniquely fix the nonlinear
connection of a Lagrange space and a non-conservative force. If $\sigma=0$, the
conditions $(LGH_1) - (LGH_3)$ represent the classic Helmholtz
conditions in terms of the multiplier matrix $g_{ij}$.

\section{The dissipative case} \label{section_dghc}

In this section we restrict our results from the previous section to the particular case when the covariant force field is
a $d_J$-closed, semi-basic $1$-form $\sigma$. Such systems were studied
in \cite{RMPS83}, with special attention on a subclass that admits a
Lagrangian description.

\begin{defn} \label{def_dis}
A semispray $S$ is said to be of \emph{dissipative type} if there
exists a Lagrangian $L$ and a $d_J$-closed semi-basic $1$-form
$\sigma$ on $TM$ such
that $\delta_SL=\sigma$.
\end{defn}  
This definition includes the classic \emph{dissipation of Rayleigh
  type}, where $\sigma=d_J{\mathcal D}$, for ${\mathcal D}$
a negative definite, quadratic function in the velocities.   

Next theorem provides three sets of equivalent conditions that
characterize dissipative systems in terms of semi-basic $1$-forms, and
it corresponds to the multiplier matrix characterisations from \cite[Corollary 1,
Theorem 1, Theorem 3]{MSC11}. The techniques we employ in the proof of
Theorem \ref{thm3} can be used to integrate the corresponding
generalized Helmholtz conditions, as it will be shown in the example
of Subsection \ref{section_ex1}.

Two of the three equivalent sets of conditions of Theorem \ref{thm3}, $(D_1)$
and $(D_3)$, do not involve the dissipative force field $\sigma$, while
the other set, $(D_2)$, does. The proof of the theorem shows how
to recover the covariant force field $\sigma$, when it is not given.     
\begin{thm} \label{thm3}
A semispray $S$ is of \emph{dissipative type} if and only if there exist a
non-trivial, semi-basic $1$-form $\theta$ and a $d_J$-closed semi-basic $1$-form
$\sigma$ on $TM$ that satisfy one of the following equivalent sets of
conditions
\begin{itemize}
\item[$(D_1)$] $d_J\theta=0$, $d_h\theta=0$; \vspace{1mm}
\item[$(D_2)$] $d_J\theta=0$, $d_{\Phi}\theta = -d_h\sigma$, $\nabla
  d_v\theta = d_v\sigma$; \vspace{1mm}
\item[$(D_3)$] $d_J\theta=0$, $d_R\theta=0$, $d_hd_v\theta=0$.   
\end{itemize}
\end{thm}
\begin{proof}
We will prove the following implications: $(D_1)$ $\Longrightarrow$ $S$
is dissipative $\Longrightarrow$ $(D_2)$ $\Longrightarrow$ $(D_3)$
$\Longrightarrow$ $(D_1)$. 

For the \emph{first implication}, we assume that there exists a
non-trivial, semi-basic $1$-form
$\theta$ such that $d_J\theta=0$ and $d_h\theta=0$. From the first
condition $(D_1)$, it follows that there exists a Lagrangian $L$, locally
defined on $TM$, such that $\theta=d_JL$. Now, the second condition
$(D_1)$ reads $0=d_hd_JL=-d_Jd_hL$. This means that the semi-basic $1$-form
$d_hL$ is $d_J$-closed and hence it is locally $d_J$-exact. Therefore, there
exists a function $f$, locally defined on $TM$, such that
$d_hL=d_Jf$. Consider now the function ${\mathcal D}=S(L)-2f$. It follows that
$d_JS(L)-2d_hL=d_J{\mathcal D}$. In view of the second expression
\eqref{elform} of $\delta_SL$, it follows that $\delta_SL=d_J{\mathcal
  D}$, which shows that the semispray $S$ is of dissipative type, with
the $d_J$-exact (and hence $d_J$-closed) semi-basic $1$-form
$\sigma=d_J{\mathcal D}$.

For the \emph{second implication}, we assume that the semispray $S$ is of
dissipative type. By definition, it follows that there exist a Lagrangian $L$ and a $d_J$-closed semi-basic $1$-form
$\sigma$ on $TM$ such
that $\delta_SL=\sigma$. Last condition implies that formula
\eqref{lsdts} is true, where $\theta=d_JL$ and $d_J\sigma=0$. Therefore
the three generalized Helmholtz conditions $(GH_1) - (GH_3)$ of Theorem
\ref{thm2} are satisfied. If we use the fact that $d_J\sigma=0$, we
have that the conditions $(GH_1) - (GH_3)$ imply the conditions $(D_2)$.

For the \emph{third implication}, we consider $\theta$ a non-trivial, semi-basic $1$-form and $\sigma$ a  $d_J$-closed semi-basic
$1$-form on $TM$ such that the three conditions $(D_2)$ are satisfied. 

Using formulae \eqref{dkl} and \eqref{rphi}, as well as the first two
conditions $(D_2)$,  it follows 
\begin{eqnarray}
3d_R\theta=d_{[J,\Phi]} \theta = d_Jd_{\Phi}\theta + d_{\Phi} d_J
\theta = - d_Jd_h\sigma = d_hd_J\sigma=0. \label{drt}
\end{eqnarray}
We apply the derivation $d_J$ to both sides of the last
condition $(D_2)$. Since $[J,v]=0$, using formula \eqref{dkl}, it follows that
$d_Jd_v\sigma = -d_vd_J\sigma=0$. Using the commutation formula
\eqref{djn}, we have
\begin{eqnarray*} 
0=d_Jd_v\sigma = d_J\nabla d_v\theta = \nabla d_J d_v\theta +
d_hd_v\theta + 2 i_Rd_v\theta. \end{eqnarray*}
In the above formula we use that $i_Rd_v\theta = i_Rd\theta = d_R\theta=0$ and
$d_Jd_v\theta =0$. It follows that $d_hd_v\theta
=0$. 

For the \emph{last implication}, we show first that $d_h\theta$ is a closed, basic $2$-form.  
The second condition $(D_3)$, $d_R\theta=0$, is equivalent to $d_hd_h\theta=0$.
Using \eqref{dkl}, we have $d_hd_v\theta + d_vd_h\theta =
d_{[h,v]}\theta = -2d_R\theta=0$. In this last formula we use the last
condition $(D_3)$ and obtain $d_vd_h\theta=0$. Since
$d_h\theta$ is a semi-basic $2$-form it follows that
\begin{eqnarray}
dd_h\theta = d_hd_h\theta + d_v d_h\theta=0. \label{ddht}
\end{eqnarray}
Above formula implies that $d_h\theta$ is a closed, basic $2$-form and
hence it is locally exact. Therefore, there exists a
basic $1$-form $\eta$ such that $d_h\theta =d \eta$. We consider now
the semi-basic $1$-form 
\begin{eqnarray} \tilde{\theta} =\theta
  -\eta. \label{te} \end{eqnarray}
Since $\eta$ is a basic $1$-form, we have that the semi-basic
$1$-forms $\tilde{\theta}$ and $\theta$ are simultaneously non-trivial. We have that $d_J \tilde{\theta} =d_J\theta = 0$ and $d_h
\tilde{\theta} =0$, which means that the non-trivial, semi-basic
$1$-form $\tilde{\theta} $ satisfies the two conditions $(D_1)$.
\end{proof}

The three conditions $(D_2)$ are equivalent to the four conditions of
\cite[Theorem 1]{MSC11}. The correspondence between the semi-basic $1$-form
$\theta=\theta_idx^i$ and the $(0,2)$-type
tensor $g=(g_{ij})$ is $g_{ij}={\partial \theta_i}/{\partial
  y^j}$. The three conditions $(D_3)$ are equivalent to the four conditions of 
\cite[Theorem 3]{MSC11}. 

Locally, the three equivalent sets of conditions $(D_1)$, $(D_2)$ and $(D_3)$ can be expressed as follows
\begin{itemize} \item[$(LD_1)$] $\displaystyle \frac{\partial \theta_i}{\partial y^j} =
    \frac{\partial \theta_j}{\partial y^i}$ , $\displaystyle \frac{\delta \theta_i}{\delta x^j} = \frac{\delta
      \theta_j}{\delta x^i}$. \vspace{1mm} \\ 
 \item[$(LD_2)$] $g_{ij}=g_{ji}$,  $\displaystyle  \frac{\partial g_{ij}}{\partial y^k} =
  \frac{\partial g_{ik}}{\partial y^j}$, $\displaystyle g_{ik}R^k_j-g_{jk}R^k_i=\frac{\delta}{\delta
      x^i}\left(\frac{\partial {\mathcal D}}{\partial y^j}\right) - \frac{\delta}{\delta
      x^j}\left(\frac{\partial {\mathcal D}}{\partial y^i}\right)$,
    $\displaystyle \nabla g_{ij}=\frac{\partial^2 {\mathcal D}}{\partial
      y^i\partial y^j}$. \vspace{1mm} \\
\item[$(LD_3)$] $g_{ij}=g_{ji}$,  $\displaystyle  \frac{\partial g_{ij}}{\partial y^k} =
  \frac{\partial g_{ik}}{\partial y^j}$, $g_{il}R^l_{jk} +
  g_{kl}R^l_{ij} + g_{jl}R^l_{ki}=0$, $g_{ij|k}-g_{ik|j}=0$.
\end{itemize} 
In the last set of conditions $(LD_3)$, $g_{ij|k} = {\delta g_{ij}}/{\delta x^k}
  - g_{il}\Gamma^l_{jk} - g_{lj}\Gamma^l_{ik}$, where
  $\Gamma^i_{jk}={\partial N^i_j}/{\partial y^k}$, is the $h$-covariant
  derivative of the tensor $g$ with respect to the Berwald connection.
The conditions $(LD_3)$ represent conditions (45)-(46) in
\cite{MSC11}, while $(LD_2)$ represent conditions (30)-(32) in
\cite{MSC11}. 

For the formal integrability of the system $(D_1)$: $d_J\theta=0$ and
$d_h\theta=0$, we can follow the results of \cite{BM11, Constantinescu12}, where the
formal integrability of very similar systems was investigated. 
Using a very similar proof as the one of Theorems 4.2 and 4.3 in
\cite{BM11} and Theorems 3 and 4 in \cite{Constantinescu12}, we can state
the following result. 
\begin{prop} \label{cor2}
The system $(D_1)$ is formally integrable if the following obstruction
is satisfied for all semi-basic $1$-forms $\theta$ 
\begin{eqnarray} d_R\theta=0. \label{drto} \end{eqnarray}
\end{prop}

An important consequence of the previous proposition appears in dimension
$2$. In this case, the obstruction $d_R\theta=0$ is automatically
satisfied, being a semi-basic $3$-form on a $2$-dimensional
manifold. Therefore, we obtain the following corollary. 
\begin{cor} \label{cor:2dim} Any semispray on a $2$-dimensional
manifold is of dissipative type. \end{cor}
The first example of Section 5 in \cite{MSC11}, as well as the example
\eqref{ex1} that we discuss in Section \ref{section_ex} agree with the conclusion of the above corollary.
 
\section{The gyroscopic case} \label{section_gghc}

In this section, we restrict the general results obtained in Section
\ref{section_ghc} to the particular case when the exterior covariant force
field $\sigma$ is of the type $i_S\omega$, for $\omega$ a basic
$2$-form. This corresponds to the gyroscopic case, studied in
\cite{MSC11}. 

\begin{defn} \label{defgyr} A semispray $S$ is said to be of
  \emph{gyroscopic type}
if there exists a Lagrangian $L$ and a basic  $2$-form $\omega$ on
$TM$ such that $\delta_{S}L=i_{S}\omega$. \end{defn}

Next theorem provides two characterisations of gyroscopic systems in
terms of a semi-basic $1$-form, and it corresponds to the multiplier
matrix characterisations from \cite[Theorem 2, Theorem4]{MSC11}.  The techniques we employ in the
proof of Theorem \ref{thm4} can be used when studying the integrability
of the corresponding generalized Helmholtz conditions, as we exemplify
it in Subsection \ref{section_ex2}.

The set of conditions $(G_2)$, of Theorem \ref{thm4},
does not involve the gyroscopic $2$-form $\omega$, while the
other set $(G_1)$ does. The proof of the theorem shows how to recover
this gyroscopic $2$-form $\omega$.

\begin{thm} \label{thm4} A semispray $S$ is of \emph{gyroscopic type} if and only if
there exist a non-trivial, semi-basic $1$-form $\theta$ and a basic $2$-form
$\omega$ on $TM$ that satisfy one of the following equivalent sets
of conditions 
\begin{itemize}
\item [$(G_{1})$] $d_{J}\theta=0$, $d_{\Phi}\theta=i_{S}d\omega$, $\nabla d_{v}\theta=0$;
\vspace{1mm}
\item [$(G_{2})$] $d_{J}\theta=0$, $d_{\Phi}\theta=i_{S}d_{R}\theta$,
$\nabla d_{v}\theta=0$. 
\end{itemize}
\end{thm}
\begin{proof}
We will prove the following implications: $S$ is gyroscopic $\Longrightarrow$ $(G_1)$ $\Longrightarrow$ $(G_2)$ $\Longrightarrow$ $S$
is gyroscopic. 

For the \emph{first implication}, we assume that $S$ is of gyroscopic type
and therefore there exists a Lagrangian $L$ and a basic $2$-form $\omega$
such that $\delta_{S}L=i_{S}\omega$.  Using third form of $\delta_SL$ in formula \eqref{elform}, we can
write the condition that $S$ is of gyroscopic type as follows
\begin{eqnarray*} \nabla d_JL- d_hL=i_S\omega. \end{eqnarray*}
If we apply the derivation $d_J$ to the both sides of the above formula, use
the commutation formula \eqref{djn}, as well as the commutation
$d_hd_J=-d_Jd_h$, we obtain
\begin{eqnarray} 2d_h d_JL +2i_Rd_JL=d_Ji_S\omega. \label{dhdjis} \end{eqnarray} 
We evaluate the right hand side of \eqref{dhdjis} by using the
commutation formula, \cite[A.1, page 205]{GM00}, 
\begin{eqnarray}
i_Xd_K=-d_Ki_X + {\mathcal L}_{KX} + i_{[K,X]}. \label{ikdx} \end{eqnarray} 
For $K=J$ and $X=S$, this commutation formula reads
\begin{eqnarray}
d_Ji_S\omega = -i_Sd_J\omega + {\mathcal L}_{\mathbb C}\omega
+i_{[J,S]}\omega = 2\omega. \label{djiso} \end{eqnarray}

We consider the non-trivial, semi-basic $1$-form $\theta=d_JL$, and use the fact
that $i_Rd_JL=0$. Then, from formula \eqref{dhdjis} it follows
$d_h\theta=\omega$ and hence ${\mathcal L}_Sd_h\theta = {\mathcal
  L}_S\omega$. The condition $\delta_SL=i_S\omega$ implies 
${\mathcal L}_Sd\theta = di_S\omega$, which can be further written as
\begin{eqnarray*}   {\mathcal L}_Sd_h\theta + {\mathcal L}_Sd_v\theta
  =  {\mathcal L}_S\omega -  i_Sd\omega \Longleftrightarrow {\mathcal L}_Sd_v\theta
 = -  i_Sd\omega. \end{eqnarray*}
In the last formula above, we use \eqref{nabla3}, \eqref{ifdvt} and
$i_{\Phi}d_v\theta = d_{\Phi}\theta$ to obtain
\begin{eqnarray*}
 \nabla d_v\theta -d_{\Phi}\theta
  = -  i_Sd\omega. \end{eqnarray*}
In both sides of this formula we separate the semi-basic $2$-forms
and obtain $d_{\Phi}\theta = i_Sd\omega$, and what remains is $\nabla
d_v\theta =0$. Therefore, we have shown that all three conditions
$(G_1)$ are satisfied. 

For the \emph{second implication}, we apply the
derivation $d_J$ to both sides of the second condition
$(G_1)$. Using the commutation formulae \eqref{djphi} and \eqref{ikdx} it
follows 
\begin{eqnarray*}
3d_R\theta = d_{[J,\Phi]}\theta = d_Jd_{\Phi}\theta = d_Ji_Sd\omega =
i_{[J,S]}d\omega =3d\omega. \end{eqnarray*}
This implies $d_R\theta=d\omega$ and hence the second condition $(G_2)$ is satisfied. 

For the \emph{last implication}, we assume now that there is a
non-trivial, semi-basic $1$-form $\theta$ that satisfies the set of conditions $(G_2)$. We apply
the derivation $d_J$ to both sides of the last condition $(G_2)$
and use the commutation formula \eqref{djn}, which implies
\begin{eqnarray*}
0=d_J\nabla d_v\theta = d_hd_v\theta + 2i_Rd_v\theta =
d_vd_h\theta. \end{eqnarray*}
Last formula expresses the fact that $d_h\theta$ is a basic
$2$-form. We denote it $d_h\theta=\omega$ and have that 
\begin{eqnarray*}
  d\omega=dd_h\theta=d^2_h\theta=d_R\theta. \end{eqnarray*}
From the first condition
$(G_2)$, and the fact that the semi-basic $1$-form $\theta$ is
non-trivial, it follows that there exists a locally defined Lagrangian $L$
such that $\theta=d_JL$. We have now $d_hd_JL=\omega$. Since $\omega$
is a basic $2$-form, it follows that $2\omega=d_Ji_S\omega$ and hence
we obtain
\begin{eqnarray*}
-d_Jd_hL=\frac{1}{2}d_Ji_S\omega  \Longleftrightarrow
d_J\left(2d_hL+i_S\omega \right)=0. \end{eqnarray*}
From the last formula it follows that there is a locally defined
function $f$ on $TM$ such that $2d_hL+i_S\omega=d_Jf$. We consider now
the function $g=S(L)-f$ and obtain 
\begin{eqnarray}
\delta_SL=d_JS(L)-2d_hL=i_S\omega + d_Jg. \label{isodjg} \end{eqnarray}
From formula \eqref{isodjg} it follows  
\begin{eqnarray}
{\mathcal L}_Sd\theta=di_Sd_h\theta + dd_Jg, \label{lsodjg} \end{eqnarray}
which can be written further as
\begin{eqnarray*}
{\mathcal L}_Sd_h\theta + {\mathcal L}_Sd_v\theta = {\mathcal
  L}_Sd_h\theta  - i_Sd_R\theta + dd_Jg. \end{eqnarray*}
Using the last two conditions of the set $(G_2)$ and the above formula, it
follows that $dd_Jg=0$. We use this in formula \eqref{lsodjg} and
obtain $ {\mathcal L}_Sd\theta=di_S\omega$, 
which in view of Theorem \ref{thm1} gives
$\delta_SL=i_S\omega$. Last formula shows that  the spray $S$ is of gyroscopic type. 
\end{proof}
The three conditions $(G_1)$, in terms of a semi-basic $1$-form, are
equivalent to the four conditions of \cite[Theorem 2]{MSC11}, expressed in terms
of a multiplier matrix. The three conditions $(G_2)$ are equivalent to
the four conditions of \cite[Theorem 4]{MSC11}.

Locally, the two sets of equivalent conditions $(G_1)$ and $(G_2)$ can be
written as follows.
\begin{itemize}
\item[$(LG_{1})$] $g_{ij}=g_{ji}$, $\displaystyle\frac{\partial g_{ij}}{\partial
    y^{k}}=\frac{\partial g_{ik}}{\partial y^{j}}$, 
 $g_{ik}R_{j}^{k}-g_{jk}R_{i}^{k}  = \displaystyle \left(
   \frac{\partial\omega_{ij}}{\partial x^{k}} +
   \frac{\partial\omega_{jk}}{\partial x^{i}} +
   \frac{\partial\omega_{ki}}{\partial x^{j}}\right)y^{k},$ $\nabla
 g_{ij}=0$, \vspace{2mm}
\item[$(LG_2)$] $g_{ij}=g_{ji}$, $\displaystyle \frac{\partial g_{ij}}{\partial
    y^{k}}=\frac{\partial g_{ik}}{\partial y^{i}}$, 
$g_{ik}R_{j}^{k}-g_{jk}R_{i}^{k}  =
\left(g_{il}R_{jk}^{l}+g_{kl}R_{ij}^{l}+g_{jl}R_{ki}^{l}\right)y^k$, $\nabla g_{ij}=0.$
\end{itemize}
Last two conditions $(LG_1)$ represent conditions
\cite[(38)-(39)]{MSC11}, while second condition $(LG_2)$ represents
condition \cite[(49)]{MSC11}. 

\section{The homogeneous case} \label{section_hghc}

In this section we study the case when all the involved geometric
objects are homogeneous with respect to the fibre
coordinates. Consequently, we will restrict all our geometric structures to the slit tangent bundle $T_0M=TM\setminus\{0\}$.

A semispray $S\in \mathfrak{X}(T_0M)$ is called a \emph{spray} if it is $2$-homogeneous,
which means that $[{\mathbb C}, S]=S$. The coefficients
$G^i(x,y)$, in formula \eqref{semispray}, which are locally defined functions on $T_0M$, are homogeneous of order $2$
in the fibre coordinates. 

\begin{defn} \label{def_ft} Consider $S$ a spray and $\sigma\in
  \Lambda^1(T_0M)$ a semi-basic $1$-form, homogeneous of order $p\in
  {\mathbb N}^{\ast}$. We say that the spray $S$ is of \emph{Finslerian type with covariant force field $\sigma$} if there exists a
  Lagrangian $L\in C^{\infty}(T_0M)$, homogeneous of order $p$, such that $\delta_SL=\sigma$.  
\end{defn} 
We will discuss now some particular cases of Definition \ref{def_ft}
in the Finslerian context. For this, we recall the notion of a Finsler
function, Finsler metrizability and projective metrizability.
 
\begin{defn} \label{finsler} A positive function $F: TM \to {\mathbb R}$ is called a \emph{Finsler
  function} if it satisfies \begin{itemize}
\item[i)] $F$ is smooth on $T_0M$ and continuous on the null section;
\item[ii)] $F$ is positively homogeneous with respect to the fibre
  coordinates: $F(x,\lambda y)=\lambda F(x,y)$, $\forall \lambda \geq
  0$;
\item[iii)] $F^2$ is a regular Lagrangian on $T_0M$.
\end{itemize}
\end{defn}
When $\sigma=0$, depending on the values of $p$,
we obtain two important particular cases of the inverse
problem of Lagrangian mechanics. The case $\sigma=0$,  $p=1$, and
$L=F$, for a Finsler function $F$, is known as the \emph{projective
metrizability problem} \cite{BM11, BM12, GM00, Klein62}. The case $\sigma=0$, $p=2$, and $L=F^2$,
for a Finsler function $F$, is known as the \emph{Finsler metrizability
problem} \cite{Muzsnay06}. For the general case, 
when $\sigma\neq 0$, we will see that we also have to make
distinction between the two cases $p=1$ and $p>1$. 

\subsection{The case $p>1$} Next theorem shows that, in the homogeneous context, the
generalized Helmholtz condition $(GH_2)$ is a consequence of the other
two generalized Helmholtz conditions $(GH_1)$ and $(GH_3)$ of Theorem \ref{thm2}.  

\begin{thm} \label{thm5} Consider $S$ a spray and $\sigma\in
  \Lambda^1(T_0M)$ a semi-basic $1$-form, homogeneous of order $p> 1$. The spray $S$ is of \emph{Finslerian type with covariant force
  field} $\sigma$ if and only if there exists a non-trivial, semi-basic form
  $\theta\in \Lambda^1(T_0M)$, homogeneous of order $(p-1)$ that satisfies the two
  generalized Helmholtz conditions $(GH_1)$ and $(GH_3)$ of Theorem \ref{thm2}. 
\end{thm}
\begin{proof}
Consider $\theta$ a non-trivial, semi-basic $1$-form, homogeneous of order $(p-1)$
that satisfies the conditions $(GH_1)$ and $(GH_3)$ of Theorem
  \ref{thm2}. We will prove that  $\delta_S(i_S\theta/p)=\sigma$. In this case, the
  Lagrangian is given by $L=i_S\theta/p$ and it is homogeneous of order $p$.

For $\theta$, satisfying the condition $(GH_1)$, we use \cite[Proposition
4.2]{BD09}. It follows that 
\begin{eqnarray} L=\frac{1}{p} i_S\theta \label{list} \end{eqnarray} is the only
$p$-homogeneous Lagrangian that satisfies $d_JL=\theta$.  

We use the same argument that we used in the proof of Theorem
\ref{thm2}, from formula \eqref{djgh3} to formula \eqref{basic3}. It
follows that the condition $(GH_3)$ implies that the $2$-form $d_h\theta
- d_J\sigma/2$ is a basic form and hence it is homogeneous of order
$0$. Since $\theta$ is homogeneous of order
$(p-1)$, and $\sigma$ is homogeneous of order
$p>1$, it follows that $d_h\theta - d_J\sigma/2$ is homogeneous of order $p-1\neq 0$. Consequently, we
obtain that this $2$-form has to vanish and therefore we have 
\begin{eqnarray}
d_h\theta=\frac{1}{2} d_J\sigma. \label{dhtf}\end{eqnarray}  
In formula \eqref{dhtf} we apply to the left the interior product
$i_S$ and use the commutation formula \eqref{ikdx} to evaluate
$i_Sd_h\theta$ and $i_Sd_J\sigma$. We have
\begin{eqnarray*}
-d_hi_S\theta + {\mathcal L}_S\theta +i_{[h,S]}\theta =
\frac{1}{2}\left(-d_Ji_S\sigma + {\mathcal L}_{\mathbb C} \sigma +
  i_{[J,S]}\sigma\right). \end{eqnarray*} 
Using the fact that $d_JL=\theta$ it follows that $i_{[h,S]}\theta = i_{J\circ
  [h, S]}dL=-d_vL$. Now, we use the $p$-homogeneity of $\sigma$ and
$i_{[J,S]}\sigma = i_{h-v}\sigma = \sigma$ and hence we obtain
\begin{eqnarray*}
-pd_hL +  {\mathcal L}_Sd_JL - d_vL = - \frac{1}{2} d_Ji_S\sigma +
\frac{p+1}{2} \sigma. \end{eqnarray*}
From the above formula we can express the Lagrange differential as
follows
\begin{eqnarray}
\delta_SL = (p-1) d_hL - \frac{1}{2} d_Ji_S\sigma +
\frac{p+1}{2} \sigma. \label{dsl1}\end{eqnarray}  
We will evaluate now the two $2$-forms in both sides of formula $(GH_3)$
on the pair of vectors $({\mathbb C}, S)$. First we have
$i_{\mathbb C}d_v\theta = i_{\mathbb C}d\theta  = {\mathcal L}_{\mathbb
  C} \theta = (p-1)\theta.$
From this formula and using also \eqref{list} we have   
\begin{eqnarray*}
i_Si_{\mathbb C}d_v\theta = (p-1)i_S\theta = p(p-1)L. \end{eqnarray*}
According to \cite[Proposition 3.6]{BD09}, in the homogeneous case,
the dynamical covariant derivative $\nabla$ commutes with the inner
products $i_S$ and $i_{\mathbb C}$. From the above formula, it follows
that the value of the $2$-form $\nabla d_v\theta$ on the pair of
vectors $({\mathbb C}, S)$ is given by 
\begin{eqnarray}
i_Si_{\mathbb C}\nabla d_v\theta = p(p-1)S(L). \label{isicn} \end{eqnarray}
Since $({\mathbb F}+J)({\mathbb C})=S$, it follows that 
$ \left( i_{{\mathbb F}+J} d_J\sigma \right) ({\mathbb C}, S)=d_J\sigma
(S,S)=0.$ Using the fact that $\sigma$ is semi-basic and $p$-homogeneous, we
have $i_{\mathbb C}d_v\sigma = i_{\mathbb C}d\sigma = {\mathcal L}_{\mathbb
  C}\sigma = p\sigma.$ It follows that the value of the $2$-form from
the right hand side of $(GH_3)$ on the pair of vectors $({\mathbb C}, S)$ is given by 
\begin{eqnarray}
i_Si_{\mathbb C}\left(d_v\sigma - \frac{1}{2}i_{{\mathbb F}+J}
  d_J\sigma\right) = pi_S\sigma. \label{isicd} \end{eqnarray}
From the two formulae \eqref{isicn} and \eqref{isicd} it follows that 
\begin{eqnarray}
(p-1)S(L) = i_S\sigma. \label{slis}\end{eqnarray} 
If we use now the commutation formula \eqref{dkl} for $J$ and $S$ we
obtain $ d_{[J,S]}L = d_J S(L) - {\mathcal L}_S d_JL$.
In this formula, we replace $S(L)$ from \eqref{slis} and use the
fact that $[J,S]=h-v$. It follows
\begin{eqnarray*}
d_hL- d_vL= \frac{1}{p-1}d_Ji_S\sigma - {\mathcal L}_S
d_JL. \end{eqnarray*}
From the above formula we can express the Lagrange differential as
follows
\begin{eqnarray}
\delta_SL = \frac{1}{p-1} d_Ji_S\sigma - 2
d_hL. \label{dsl2}\end{eqnarray}  
From the two formulae \eqref{dsl1} and \eqref{dsl2} it follows that
$\delta_SL=\sigma$ and hence the spray $S$ is of Finslerian type with covariant force field $\sigma$.
\end{proof}
When $\sigma=0$, Theorem \ref{thm5} reduces to \cite[Lemma 3.4]{Prince08} and \cite[Theorem 3.4]{Rossi14}, where it has been
shown that the Helmholtz condition which involves the Jacobi endomorphism is a consequence of the other
Helmholtz conditions. 

From Theorem \ref{thm5}, in the particular case when $\sigma=0$, we obtain the following
corollary that provides a characterization of the Finsler
metrizability problem in terms of a semi-basic, $1$-homogeneous, 
$1$-form.
\begin{cor}  \label{cor3} A spray $S$ is Finsler metrizable if and
  only if there exists a semi-basic $1$-form $\theta\in
  \Lambda^1(T_0M)$ that satisfies
  the following sets of conditions
\begin{itemize}
\item[(FMA)] $d\theta$ is a symplectic form, $i_S\theta>0$; 
\item[(FMD)] $d_J\theta=0$, $\nabla d_v\theta=0$, ${\mathcal
    L}_{\mathbb C}\theta=\theta.$
\end{itemize}
\end{cor}
The first two differential conditions, $(\textrm{FMD})$, for Finsler
metrizability, represent the two Helmholtz conditions $(H_1)$ and $(H_3)$
of Corollary \ref{cor1}.

\subsection{The case $p=1$}

Next theorem gives a reformulation of the generalised Helmholtz
conditions in the presence of a covariant force field $\sigma$, homogeneous of order $1$. In this
case, the semi-basic $1$-form $\sigma$ has to satisfy the
condition $i_S\sigma=0$, due to the following argument. Consider a spray $S$ and a
$1$-homogeneous, semi-basic $1$-form $\sigma$. We assume that there is a
$1$-homogeneous Lagrangian $L$ such that $\delta_SL=\sigma$. If we
apply $i_S$ to both sides of this formula we obtain $S\left({\mathbb
    C}(L)-L\right)=i_S\sigma$ and hence we necessarily should have $i_S\sigma=0$.  

\begin{thm} \label{thm7}
Consider $S$ a spray and a semi-basic $1$-form $\sigma\in
  \Lambda^1(T_0M)$, homogeneous
of order $1$ that satisfies the condition $i_S\sigma=0$. The spray $S$
is of Finslerian type with
covariant force field $\sigma$ if and only if there exists a
non-trivial, $0$-homogeneous semi-basic $1$-form $\theta\in
  \Lambda^1(T_0M)$ that satisfies one of the
following two equivalent sets of conditions
\begin{itemize}
\item[i)] $d_J\theta=0$, $d_h\theta=\displaystyle\frac{1}{2}d_J\sigma$; \vspace{2mm}
\item[ii)] $(GH_1)$, $(GH_2)$, $(GH_3)$.
\end{itemize}
\end{thm}
\begin{proof}
In view of Theorem \ref{thm2}, it remains to prove the equivalence of
the two sets of conditions. 

First we assume that there exists a non-trivial, 
$0$-homogenous, semi-basic $1$-form $\theta$ that satisfies i). Since $d_J\theta=0$, using \cite[Proposition 4.2]{BD09}
it follows that $L=i_S\theta$ is the only $1$-homogeneous Lagrangian
such that $\theta=d_JL$. The second condition i) can be written as
$d_J\left(2d_hL+\sigma\right)=0$, which implies that there exists a
  locally defined $2$-homogeneous function $f$ such that $2d_hL+\sigma = -d_Jf$. We
  consider now the $2$-homogeneous function $g=f-S(L)$. It follows
  $\delta_SL=\sigma + d_Jg$. We apply $i_S$ to both sides of the
  last formula and obtain $0={\mathbb C}(g)=2g$. Therefore
  $\delta_SL=\sigma$, which in view of Theorem \ref{thm2} implies that
  the conditions ii) are satisfied.  

For the converse, assume that there exists a non-trivial, $0$-homogenous, semi-basic $1$-form
$\theta$ that satisfies the three condition $(GH_1) - (GH_3)$. As we have
seen in the proof of Theorem \ref{thm2}, the two conditions $(GH_2)$ and
$(GH_3)$ imply that there exists a basic $1$-form $\beta$ such that
formula \eqref{basic4} is satisfied. The $0$-homogenous, semi-basic
$1$-form $\tilde\theta=\theta-\beta$ is non-trivial, satisfies the set of conditions
$i)$ and this completes the proof. 
\end{proof}
When $\sigma=0$, the two conditions i) were used to characterize the
projective metrizability of a spray in \cite[Theorem
3.8]{BM11}. Again, when $\sigma=0$, the three conditions ii) were
expressed in terms of the angular metric of some Finsler function as
classic Helmholtz conditions in \cite{CMS13} and their equivalence
with the set of conditions i) was proven in \cite[Theorem 4]{CMS13}. 

In the particular case when the covariant force field $\sigma$ is of
gyroscopic type, we obtain, using Theorem \ref{thm7}, the following
characterization for Finslerian gyroscopic sprays.
\begin{cor} \label{cor5} A spray $S$ is of Finslerian gyroscopic type
  if and only if there exists a non-trivial, semi-basic $1$-form
  $\theta\in \Lambda^1(T_0M)$ that satisfies the following conditions 
\begin{eqnarray}
{\mathcal L}_{\mathbb C}\theta=0, \quad d_J\theta=0, \quad d_h\theta=\omega, \label{fgt} \end{eqnarray}
for $\omega$ a basic $2$-form.
\end{cor}
\begin{proof}
If $\omega$ is a basic $2$-form, using formula \eqref{djiso}, it
follows that $d_Ji_S\omega=2\omega$. Therefore, the last
condition i) of Theorem \ref{thm7} is equivalent to the last condition \eqref{fgt}.
\end{proof}
As we have seen in the proof of Theorem \ref{thm4}, the last two conditions
\eqref{fgt} are necessary conditions for an arbitrary semispray to be of
gyroscopic type. Above Corollary \ref{cor5} states that in the
homogeneous case, these two conditions are also sufficient.

In the particular case when $\omega=0$, the equations \eqref{fgt}
coincide with the differential
equations \cite[(3.8)]{BM11}, which together with some
algebraic equations, provide a characterization of projectively metrizable sprays.

\section{Examples} \label{section_ex}

\subsection{Non-variational, projectively metrizable sprays that are
  dissipative} \label{section_ex1}

In this subsection we will provide examples of projectively metrizable
sprays, which are not Finsler metrizable, and hence not variational but are of dissipative type.

A spray $S$ is \emph{projectively metrizable} if there exists a Finsler
function $F$ such that $\delta_SF=0$. We underline the fact that $F$ is a Finsler function
and it is not a regular Lagrangian. However, for a Finsler function
$F$, we have that $L=F^2$ is a regular Lagrangian. A spray $S$ is
\emph{Finsler metrizable} if there exists a Finsler
function $F$ such that $\delta_SF^2=0$. In this case, we say that $S$
is the geodesic spray of the Finsler function $F$.

In \cite{BM12} it has been shown that for a given geodesic spray, its projective class contains
infinitely many sprays that are not Finsler metrizable. More exactly,
in \cite[Theorem 5.1]{BM12} it has been shown that if $S_F$ is the geodesic
spray of a Finsler function $F$, then there are infinitely many values
of $\lambda \in {\mathbb R}$ such that the projectively related sprays $S=S_F-2\lambda
F{\mathbb C}$ are not Finsler metrizable. We consider such a spray
$S$, which is projectively metrizable and hence satisfies
$\delta_SF=0$, and we will prove that the spray $S$  is of dissipative
type. Using the fact that $S_F$ is the geodesic spray of $F$ we have
that $S_F(F)=0$ and hence $S(F)=S_F(F)-2\lambda F{\mathbb
  C}(F)=-2\lambda F^2$. Therefore 
\begin{eqnarray}
\delta_SF^2=2S(F)d_JF + 2F\delta_SF= 2S(F)d_JF=-4\lambda F^2 d_JF. \label{dispm}\end{eqnarray}
The semi-basic $1$-form $\sigma = -4\lambda F^2 d_JF$ is $d_J$-closed and from 
formula \eqref{dispm} it follows that the spray $S$ is of
dissipative type. 

Within the class of sprays that we have discussed above, we will
consider now a concrete example in dimension $2$. We consider the spray $S\in {\mathfrak
  X}({\mathbb R}^2 \times {\mathbb R}^2)$, given by 
\begin{eqnarray}
S=y^1\frac{\partial}{\partial x^1} + y^2\frac{\partial}{\partial x^2}
- \left( (y^1)^2 + (y^2)^2 \right) \frac{\partial}{\partial y^1} -
  4y^1y^2 \frac{\partial}{\partial y^2}. \label{ex1}\end{eqnarray} 
In \cite[Example 7.2]{AT92}  it has been shown
that the system of second order ordinary differential equations
corresponding to this spray is not variational. Since it is a $2$-dimensional spray, it follows that $S$
is projectively metrizable. According to Corollary \ref{cor:2dim}, it
follows that the sprays $S$ is of dissipative type. We will show here
directly, that the spray $S$ is of
dissipative type by using the characterization $(D_1)$ of Theorem \ref{thm3}. We will prove that there exists a semi-basic $1$-form
$\theta=\theta_{1}(x,y)dx^1 + \theta_2(x,y)dx^2$ that satisfies the
two conditions $(D_1)$, which can be written as follows
\begin{eqnarray}
\frac{\partial \theta_1}{\partial y^2}= \frac{\partial \theta_2}{\partial y^1},
\quad \frac{\delta \theta_1}{\delta x^2}= \frac{\delta
  \theta_2}{\delta x^1}. \label{d1ex1}
\end{eqnarray}
For the given spray \eqref{ex1}, the coefficients of the canonical nonlinear
connection are given by
\begin{eqnarray*}
N^1_1=y^1, \ N^1_2= y^2, \ N^2_1=2y^2, \ N^2_2=2y^1. \end{eqnarray*}
Using first condition \eqref{d1ex1}, the second condition
\eqref{d1ex1} can be written as follows
\begin{eqnarray}
\frac{\partial \theta_1}{\partial x^2} - \frac{\partial
  \theta_2}{\partial x^1} = y^2\left( \frac{\partial
    \theta_1}{\partial y^1} - 2\frac{\partial \theta_2}{\partial
    y^2}\right) + y^1\frac{\partial \theta_1}{\partial
    y^2}. \label{gsolex1}\end{eqnarray}
 It is easy to see that the PDE system \eqref{gsolex1} has solutions. For
 example $ \theta_1(x^1, y^1)dx^1 + \theta_2(x^2, y^2)dx^2 $ is a solution to the system
\eqref{gsolex1} if and only if $\partial \theta_1/\partial
y^1=2\partial \theta_2/\partial y^2$. A Riemannian solution to
this system is provided by $\partial \theta_1/\partial
y^1=2\partial \theta_2/\partial y^2=2$. Such a semi-basic
$1$-form $\theta$, which is also homogeneous of order $1$, is given by
$\theta=2y^1dx^1 + y^2dx^2$. Using formula \eqref{list}, from the proof
of Theorem \ref{thm5}, it follows that the corresponding Lagrangian
can be obtained as
$L=i_S\theta$,
\begin{eqnarray}
L=\frac{1}{2}\left(2(y^1)^2+
  (y^2)^2\right) \label{lex1}.  \end{eqnarray}
We know that system \eqref{ex1} is dissipative since we found a
semi-basic $1$-form $\theta$ that satisfies the conditions $(D_1)$ of
Theorem \ref{thm4}. Using the calculations we did in the proof of
Theorem \ref{thm4} (first implication), we obtain that the dissipative
function is given by 
\begin{eqnarray} 
{\mathcal D}=-\frac{2}{3}(y^1)^3 - 2y^1(y^2)^2. \label{dex1} \end{eqnarray}
One can also check directly that for the spray $S$, given by formula
\eqref{ex1}, the Lagrangian $L$, given by formula \eqref{lex1}, and the dissipative function
${\mathcal D}$, given by formula \eqref{dex1}, we have
$\delta_SL=d_J{\mathcal D}$, which means that $S$ is of dissipative type.
   
\subsection{A class of gyroscopic semisprays} \label{section_ex2}

Historically, gyroscopic systems are systems of second order ordinary
differential equations in ${\mathbb R}^n$, of the form
\begin{eqnarray}
\frac{d^2x^i}{dt^2} = A^i_j \frac{dx^j}{dt} + B^i_jx^j, \label{classicgyr}\end{eqnarray}
where $A^i_j$ is a constant, skew symmetric matrix and $B^i_j$ is a constant, symmetric matrix.

We will consider now a generalization of the above system, on some
open domain $\Omega\subset {\mathbb R}^n$,
\begin{eqnarray}
\frac{d^2x^i}{dt^2} + 2N^i_j(x) \frac{dx^j}{dt} + V^i(x)=0. \label{exgyr}\end{eqnarray}
Consider $g_{ij}$ a scalar product in ${\mathbb R}^n$. Using the
conditions $(G_1)$ of the Theorem \ref{thm4}, we will
prove that the system \eqref{exgyr} is of gyroscopic type, with the
Lagrangian function $L(x,y)=g_{ij}y^iy^j/2$, if and only if the
functions $N^i_j(x)$ and $V^i(x)$ satisfy the following conditions
\begin{eqnarray}
g_{ik}N^k_j+ g_{jk}N^k_i=0, \quad g_{ik}\frac{\partial V^k}{\partial
  x^j} - g_{jk}\frac{\partial V^k}{\partial x^i} =
0. \label{gnv} \end{eqnarray}
In the particular case when $g_{ij}=\delta_{ij}$ is the Euclidean
metric, the functions $N^i_j$ are constant and $V^i(x)$ are linear functions, the
conditions \eqref{gnv}  assure that \eqref{classicgyr} is a gyroscopic
system.

We will use the local version $(LG_1)$ of the set of conditions $(G_1)$
of Theorem \ref{thm4} to test when the system \eqref{exgyr} is of gyroscopic type. 
For $L=g_{ij}y^iy^j/2$, its Poincar\'e-Cartan
$1$-form is $\theta=d_JL=g_{ij}y^jdx^i$. Therefore, we want to find the
necessary and sufficient conditions for the existence of a basic
$2$-form $\omega\in \Lambda^2(\Omega)$ that satisfies the
conditions $(G_1)$ of Theorem \ref{thm4}. As we have seen in the proof
of Theorem \ref{thm4}, the basic $2$-form $\omega$ is necessarily
given by $d_h\theta=\omega$. Locally, the
components $\omega_{ij}$ of the $2$-form $\omega$ are given by 
\begin{eqnarray}
\omega_{ij}=N_i^kg_{kj}- N_j^kg_{ki}. \label{f1o} \end{eqnarray}
We pay attention now to the second condition $(LG_1)$, which reads
$\nabla g_{ij}=0$ and implies 
\begin{eqnarray}
N_i^kg_{kj} + N_j^kg_{ki}=0, \label{f2o} \end{eqnarray}
which is first condition \eqref{gnv}. From the two formulae
\eqref{f1o} and \eqref{f2o} it follows 
\begin{eqnarray}
N^i_j(x)=\frac{1}{2}g^{ik}\omega_{jk}(x). \label{f3o}
\end{eqnarray}
We use formula \eqref{rij} to compute the local components $R^i_j$ of
the Jacobi endomorphism,
\begin{eqnarray*}
R^k_j = 2\frac{\partial N^k_l}{\partial x^j} y^l + \frac{\partial
  V^k}{\partial x^j} - \frac{\partial N^k_j}{\partial x^l}y^l -
N^k_lN^l_j. \end{eqnarray*}
From this formula, and using formula \eqref{f3o}, it follows 
\begin{eqnarray}
g_{ik}R^k_j-g_{kj}R^k_i = \left(\frac{\partial \omega_{ij}}{\partial
    x^l} + \frac{\partial \omega_{li}}{\partial x^j} + \frac{\partial \omega_{jl}}{\partial
    x^i} \right) y^l + g_{ik} \frac{\partial V^k}{\partial x^j} -
g_{jk} \frac{\partial V^k}{\partial x^i}. \label{f4o}   \end{eqnarray} 
Therefore, last condition $(LG_1)$ is satisfied if and only if the
second condition \eqref{gnv} is satisfied. It follows that the system
\eqref{exgyr} is gyroscopic if and only if the two conditions \eqref{gnv} are satisfied.

\begin{acknowledgement*}
The authors acknowledge fruitful discussions with Tom Mestdag and
Willy Sarlet. This work has been supported by the Bilateral Cooperation Program Ro-Hu 672/2013-2014.
\end{acknowledgement*}

\end{document}